\documentclass{amsart}

\usepackage{amssymb}
\usepackage{amsfonts}
\usepackage{amsmath}
\usepackage{amsthm}
\usepackage{url}
\usepackage{color}
\usepackage[small,nohug,heads=LaTeX]{diagrams}
\diagramstyle[labelstyle=\scriptstyle]
\usepackage{graphicx}

\newcommand{\bq}{\begin{quote}}
\newcommand{\eq}{\end{quote}}
\newcommand{\bi}{\begin{itemize}}
\newcommand{\ei}{\end{itemize}}
\newcommand{\bd}{\begin{description}}
\newcommand{\ed}{\end{description}}
\newcommand{\ben}{\begin{enumerate}}
\newcommand{\een}{\end{enumerate}}
\newcommand{\bbm}{\begin{bmatrix}}
\newcommand{\ebm}{\end{bmatrix}}
\newcommand{\bea}{\begin{eqnarray*}}
\newcommand{\eea}{\end{eqnarray*}}

\newtheorem{theorem}{Theorem}[section]

\def\sX{\mathcal{X}}

\def\GF{\mathbb{F}}

\def\CC{\mathbf{C}}

\def\KK{\mathbf{K}}
\def\LL{\mathbf{L}}

\def\TT{\mathbf{T}}
\def\UU{\mathbf{U}}
\def\VV{\mathbf{V}}
\def\WW{\mathbf{W}}
\def\XX{\mathbf{X}}
\def\YY{\mathbf{Y}}
\def\ZZ{\mathbf{Z}}

\def\Aut{\mathop{{\rm Aut}}}

\def\2G2{\ensuremath{^2{\rm G}_2}}
\def\sl{{\rm{SL}}}
\def\gl{{\rm{GL}}}

\def\su{{\rm{SU}}}
\def\psl{{\rm{PSL}}}
\def\pgl{{\rm{PGL}}}

\def\so{{\rm{SO}}}

\def\ppsl{ ( {\rm{P}} ) {\rm{SL}}}

\def\ppsp{ ( {\rm{P}} ) {\rm{Sp}}}

\definecolor{darkgreen}{rgb}{0,0.6,0}

\begin{document}

\title{Black box, white arrow}
\date{29 April 2014}

\author{Alexandre Borovik}
\address{School of Mathematics, University of Manchester, UK}
\address{alexandre@borovik.net}
\author{\c{S}\"{u}kr\"{u} Yal\c{c}\i nkaya}
\address{\.{I}stanbul University, previously Nesin Mathematics Village, Izmir, Turkey}
\address{sukru.yalcinkaya@gmail.com}


\begin{abstract}
The present paper proposes a new and systematic approach to the so-called black box group methods in computational group theory. Instead of a single black box, we consider categories of black boxes and their morphisms. This makes new classes of black box problems accessible. For example, we can enrich  black box groups by actions of outer automorphisms.

As an example of application of this technique, we construct Frobenius maps on black box groups of untwisted Lie type in odd characteristic (Section~\ref{ssec:Frobenius}) and inverse-transpose automorphisms on black box groups encrypting $\ppsl_n(\GF_q)$.

One of the advantages of our approach is that it allows us to work in black box groups over finite fields of big characteristic. Another advantage is  explanatory power of our methods; as an example, we explain Kantor's and Kassabov's construction of an involution in black box groups encrypting $\sl_2(2^n)$.

Due to the nature of our work we also have to discuss a few methodological issues of the black box group theory.\\

\noindent

\end{abstract}

\maketitle

\begin{footnotesize}
\tableofcontents
\end{footnotesize}

\section{Introduction}

Black box groups were introduced by Babai and Szemeredi \cite{babai84.229} as an idealized setting for randomized algorithms for solving permutation and matrix group problems in computational group theory. A black box group $\XX$ is a black box (or an oracle, or a device, or an algorithm) operating with $0$--$1$ strings of uniform length which encrypt (not necessarily in a unique way) elements of some finite group $G$. In various classes of black box problems the isomorphism type of $G$ could be known in advance or unknown.

This original definition of a black box group appears to have gone out of vogue with the black box group community mostly because it is too abstract and demanding in the context of practical computations with matrix groups. The aim of this paper is to refresh and expand the concept by introducing and systematically using \emph{morphisms}, that is, polynomial time computable homomorphisms of black box groups.

This modest change allows us to apply some basic category theoretic and model theoretic ideology.
As we show in this and subsequent papers, this dramatically expands the range of available tools for structural analysis of black box groups and allows to solve black box problems previously not accessible to existing methods. The development of this new approach requires a detailed discussion of some methodological issues of black box group theory.

The  paper contains a number of new results, and two of them deserve highlighting. They are concerned with construction of automorphisms of black box groups; the first one is construction of Frobenius maps on black box Chevalley groups of untwisted type and odd characteristic. This is stated and proven in Section~\ref{ssec:Frobenius}.

The second---and the most important---result of the paper is the ``reification of involutions'', Theorem~\ref{th:reification}.

In our next work \cite{suko14A}, these constructions are applied to recognition of black box groups $\ppsl_2(q)$ for odd $q$.

Finally, we have to mention that the paper belongs to a series of works aimed at a systematic development of methods of structural analysis of black box groups \cite{suko03,suko12A,suko14B,suko14A,suko12B,suko12C,suko02,yalcinkaya07.171}.

\section{Black box groups}
\subsection{Axioms for black box groups}

The functionality of a black box $\XX$ for a finite group $G$ is specified by the following axioms.

\begin{itemize}
\item[\textbf{BB1}] $\XX$ produces  strings of fixed length $l(\XX)$ encrypting random (almost) uniformly distributed elements from $G$; this is done in probabilistic time polynomial in $l(\XX)$.
\item[\textbf{BB2}] $\XX$ computes, in probabilistic time polynomial in $l(\XX)$, a string encrypting the product of two group elements given by strings or a string encrypting the inverse of an element given by a string.
\item[\textbf{BB3}] $\XX$ decides, in probabilistic time polynomial in $l(\XX)$, whether two strings encrypt the same element in $G$---therefore identification of strings is a canonical projection
\begin{diagram} \XX & \rDotsto^\pi & G.
 \end{diagram}
\end{itemize}

We shall say in this situation that $\XX$ is a \emph{black box over $G$} or that a black box $\XX$ \emph{encrypts} the group $G$. Notice that we are not making any assumptions of practical computability or the time complexity of the projection $\pi$.

A typical example of a black box group is provided by a group $G$ generated in a big matrix group $\gl_n(r^k)$ by several matrices $g_1,\dots, g_l$. The product replacement algorithm \cite{celler95.4931} produces  a sample of (almost) independent elements from a distribution on $G$ which is close to the uniform distribution (see a discussion and further development in \cite{babai00.627, babai04.215, bratus99.91, gamburd06.411, lubotzky01.347, pak00.476, pak01.301, pak.01.476, pak02.1891}). We can, of course, multiply, invert, compare matrices. Therefore the computer routines for these operations together with the sampling of the product replacement algorithm run on the  tuple of generators $(g_1,\dots, g_l)$ can be viewed as a black box $\XX$ encrypting the group $G$. The group $G$ could be unknown---in which case we are interested in its isomorphism type---or its isomorphism type could be known, as it happens in a variety of other black box problems.

The concept of a black box can be applied to rings, fields, and, in our next paper \cite{suko14A}, even to projective planes. We shall construct new black boxes from the given ones, and in these constructions strings in $\XX$ will actually be  pointers to other black boxes. Therefore it is convenient to think of elements of black boxes as other black boxes---the same way as in the ZF set theory all objects are sets, with some sets being elements of others. A projective plane constructed in our next paper \cite{suko14A} provides a good example: it could be seen as consisting of points and lines, where a ``line'' is a black box that produces random ``points'' on this line and a ``point'' is a black box that produces random ``lines'' passing through this point.

By the nature of our axioms, all algorithms for black box groups (in the sense of Axioms BB1--BB3) are Monte Carlo. In most applications, they can be easily made Las Vegas if additional information of some kind is provided about $\XX$---for example a set of its generators, that is, strings in $\XX$ that represent a generating set of $G$, or the isomorphism type of $G$.

In our subsequent papers, especially in \cite{suko14A}, it becomes clear that the distinction between Monte Carlo  and Las Vegas probabilistic algorithms is external to the theory of black box groups although it is quite natural in its concrete applications.

\subsection{Global exponent and Axiom BB4}

Notice that even in routine examples the number of elements of a matrix group $G$ could be astronomical, thus making many natural questions about the black box $\XX$ over $G$---for example, finding the isomorphism type or the order of $G$---inaccessible for all known deterministic methods. Even when $G$ is cyclic and thus is characterized by its order, existing approaches to finding exact multiplicative orders of matrices over finite fields are conditional and involve oracles either for the discrete logarithm problem in finite fields or for prime factorization  of integers.

Nevertheless black box problems for matrix groups have a feature which makes them more accessible:

\begin{itemize}
\item[\textbf{BB4}] We are given a \emph{global exponent} of $\XX$, that is, a natural number $E$ such that $\pi(x)^E = 1$ for all strings $x \in \XX$ while computation of $x^E$ is computationally feasible (say, $\log E$ is polynomially bounded in terms of $\log |G|$).
\end{itemize}

Usually, for a black box group $\XX$ arising from a subgroup in the ambient group $\gl_n(r^k)$, the exponent of $\gl_n(r^k)$ can be taken for a global exponent of $\XX$.

One of the reasons why  the axioms BB1--BB4, and, in particular, the concept of global exponent, appear to be natural, is provided by some surprising model-theoretic analogies. For example, D'Aquino and Macintyre  \cite{D'Auquino00.311} studied non-standard finite fields defined in a certain fragment of bounded Peano arithmetic; it is called $I\Delta_0+\Omega_1$ and imitates proofs and computations of polynomial time complexity in modular arithmetic. It appears that such a basic and fundamental fact as the Fermat Little Theorem has no proof that can be encoded in $I\Delta_0+\Omega_1$; the best that has so far been proven in $I\Delta_0+\Omega_1$ is that the multiplicative group $\GF_p^*$ of the prime field $\GF_p$ has a global exponent $E <2p$ \cite{D'Auquino00.311}. Under a stronger version of bounded arithmetic, the Fermat Little Theorem is proven by  Je{\v{r}}{\'a}bek \cite{jerabek10.262}. These results have to be seen in the context of the earlier work by Kraj{\'{\i}}{\v{c}}ek  and Pudl{\'a}k
\cite{krajicek98.82} who proved that Buss' subtheory $S^1_2$ does not prove the Fermat Little Theorem if the RSA system is secure. Another result by Je{\v{r}}{\'a}bek \cite{jerabek:dual-wphp} directly links bounded arithmetic with the black box group theory: he showed that the Fermat Little Theorem in $S^1_2$ equivalent to the correctness of the Rabin-Miller primality testing \cite{rabin80.128}, an archetypal example of a black box group algorithm.

We shall discuss model theory and logic connections of black box group theory in some detail elsewhere.

\subsection{Cartan decomposition  and Axiom BB5}

Our last comment on the axiomatics of black box groups is an observation that in almost all our work in subsequent papers  Axiom BB4 can be replaced by its corollary, Axiom BB5,  the latter is closely connected to the concept of Cartan decomposition in Lie groups, see the discussion of Cartan decomposition in \cite{borovik-inherited}.

\begin{itemize}
\item[\textbf{BB5}] We are given a partial $1$- or $2$-valued function $\rho$ of two variables on $\XX$ that computes, in probabilistic time polynomial in $l(\XX)$, square roots in cyclic subgroups of $\XX$ in the following sense:
    \bq
    if $x\in \XX$ and $y\in \langle x\rangle$ has square roots in $\langle x\rangle$ then $\rho(x,y)$ is the set of these roots.
    \eq
\end{itemize}

In particular,
\bi
\item
if $|x|$ is even, $\rho(x,1)$ is the subgroup of order $2$ in $\langle x \rangle$;
\item if $|x|$ is even, then, consecutively applying $\rho(x, \cdot)$ to $2$-elements in $\langle x \rangle$, we can find all $2$-elements in $\langle x \rangle$;
\item if $|x|$ is odd, and $y\in \langle x\rangle$ then $\rho(x,y)$ is the unique square root of $y$ in $\langle x \rangle$.
\ei

We emphasize that Axiom BB5 provides everything needed for construction of centralizers of involutions by the maps $\zeta_0$ and $\zeta_1$ \cite{borovik02.7}.

Axiom BB5 follows from BB4 by the  Tonelli-Shanks algorithm \cite{shanks73.51,tonelli1891.344} applied to the cyclic group $\langle x \rangle$.

\begin{quote}
\textbf{\emph{In this paper, we assume that all our black box groups satisfy assumptions BB1--BB4 or BB1--BB3 and B5. }}
\end{quote}

We emphasize that we do not assume that black box groups under consideration in this paper are given as subgroups of ambient matrix groups; thus our approach is wider than the setup of the computational matrix group project \cite{leedham-green01.229}. Notice that we are not using the Discrete Logarithm Oracles for finite fields $\GF_q$: in our setup, we start with a black box group without any access to the field over which the group is defined. Nevertheless we are frequently concerned with black box groups encrypting classical linear groups; even so, some of our results (such as Theorems~\ref{th:sun-in-sln} and \ref{th:G2-in-SL8}) do not even involve the assumption that we know the underlying field of the group but instead assume that we  know the characteristic of the field without imposing bounds on the size of the field. Finally, in the case of groups over fields of small characteristics we can prove much sharper results, see, for example, \cite{suko12A}. Here, it is natural to call the characteristic $p$ ``small'' if it is known and small enough for the linear in $p$ running time of algorithms to be feasible.

So we attach to statements of our results one of the two labels:
\bi
\item known characteristic, or
\item small characteristic.
\ei
Our next paper \cite{suko12B} is dominated by ``known characteristic'' results. In this one, we also obtain some more specific results for small odd characteristics.

\section{Morphisms}

\subsection{Morphisms}
\label{sec:morphisms}

Given two  black boxes $\XX$ and $\YY$ encrypting  finite groups $G$ and  $H$, respectively, we say that a map $\zeta$ which assigns strings from $\YY$ to strings from $\XX$  is a \emph{morphism} of black box groups, if
\bi
\item the map $\zeta$ is computable in probabilistic time polynomial in $l(\XX)$ and $l(\YY)$, and
\item there is an abstract homomorphism $\phi:G \to H$ such that the following diagram  is commutative:
\begin{diagram}
\XX &\rTo^{\zeta} &\YY\\
\dDotsto_{\pi_{\XX}} & &\dDotsto_{\pi_{\YY}}\\
G &\rTo^{\phi} & H
\end{diagram}
where $\pi_\XX$ and $\pi_\YY$ are the canonical projections of $\XX$ and $\YY$ onto $G$ and $H$, respectively.
\ei
We shall say in this situation that a morphism $\zeta$ \emph{encrypts} the homomorphism $\phi$. For example, morphisms arise naturally when we replace a generating set for the black box group $\XX$ by a more convenient one and start sampling the product replacement algorithm for the new generating set; in fact, we replace a black box for $\XX$  and deal with a morphism $\YY \longrightarrow \XX$ from the new black box $\YY$ into $\XX$.

Observe that a map \begin{diagram}
G & \rDotsto^{\phi} & H
\end{diagram} from a group to a group is a homomorphism of groups if and only if its graph
 \[
F = \{(g,\phi(g)): g\in G\}
\]
is a subgroup of $G \times H$.

At this point it becomes useful to introduce direct products of black boxes: if $\XX$ encrypts $G$ and $\YY$ encrypts $H$ then the black box  $\XX\times \YY$ produces pairs of strings $(x,y)$ by sampling $\XX$ and $\YY$ independently, with operations carried out componentwise in $\XX$ and $\YY$; of course, $\XX\times \YY$ encrypts $G\times H$.

This allow us to treat a morphism
\begin{diagram}
\XX & \rDotsto^{\zeta} & \YY
\end{diagram}
of black box groups as a black box subgroup $\ZZ\hookrightarrow \XX\times \YY$ encrypting $F$:
\[
\ZZ = \{(x,\zeta(x)): x\in \XX\}
\]
with the natural projection
\bea
\pi_\ZZ: \ZZ & \longrightarrow & F\\
(x,\zeta(x)) & \mapsto & (\pi_\XX(x), \phi(\pi_\XX(x)).
\eea

In practice this could  mean (although in some cases we use a more sophisticated construction) that we may be able to find strings $x_1,\dots,x_k$ generating $\XX$ with known to us images $y_1=\zeta(x_1),\dots, y_k=\zeta(x_k)$ in $\YY$ and then use the product replacement algorithm to run a black box for the subgroup \[\ZZ = \langle (x_1,y_1),\dots,(x_k,y_k)\rangle \leqslant \XX \times \YY\] which is of course exactly the graph $\{\, (x,\zeta(x))\,\}$ of the homomorphism $\zeta$. Random sampling of the black box $\ZZ$ returns strings $x\in \XX$ with their images $\zeta(x)\in \YY$ already attached.

Slightly abusing terminology, we say that a morphism $\zeta$ is an embedding, or an epimorphism, etc., if  $\phi$ has these properties. In accordance with standard conventions, hooked arrows \begin{diagram}& \rInto & \end{diagram} stand for embeddings and double-headed arrows \begin{diagram}& \rOnto & \end{diagram} for epimorphisms; dotted arrows are reserved for abstract homomorphisms, including natural projections
\begin{diagram}
\XX & \rDotsto^{\pi_\XX} & \pi(\XX);
\end{diagram}
the latter are not necessarily morphisms, since, by the very nature of black box problems, we do not have efficient procedures for constructing the projection of a black box onto the (abstract) group it encrypts.

\subsection{Shades of black}

Polynomial time complexity is an asymptotic concept, to work with it we need an infinite class of objects. Therefore our theory refers to some infinite family $\sX$ of black box groups ($\sX$ of course varies from one black box problem to another). For $\XX \in \sX$, we denote by $l(\XX)$ the length of $0$--$1$ strings representing elements in $\XX$. We assume that, for every $\XX\in \sX$, basic operations of generating, multiplying, comparing strings in $\XX$ can be done in probabilistic polynomial time in $l(\XX)$.

We also assume that the lengths $\log E(\XX)$ of  global exponents $E(\XX)$ for $\XX\in \sX$ are bounded by a polynomial  in $l(\XX)$.

Morphisms $\XX \longrightarrow \YY$ in $\sX$ are understood as defined in Section~\ref{sec:morphisms} and their running times are bounded by a polynomial in $l(\XX)$ and $l(\YY)$.

At the expense of {slightly increasing $\sX$ and its bounds for complexity, we can include in $\sX$ a collection of explicitly given ``known'' finite groups. Indeed, using standard computer implementations of finite field arithmetic, we can represent every group $Y = \gl_n(p^k)$ as an algorithm or computer routine operating on $0$--$1$ strings of length $l(Y)= n^2k\lceil \log p \rceil$. Using standard matrix representations for simple algebraic groups, we can represent every group of points $Y = {\rm G}(p^k)$ of a reductive algebraic group $\rm G$ defined over $\GF_{p^k}$ as a black box $\YY$ generating and processing strings of length $l(\YY)$ polynomial in $k \log p$ and the Lie rank of $\YY$. Therefore an ``explicitly defined'' group can be seen a black box group, perhaps of a lighter shade of black.

We  feel that the best way to analyze a black box group $\XX$ encrypting a finite group $G$
is by a step-by-step construction of a chain of morphisms
\begin{diagram}
G & \lDotsto & \XX & \lOnto & \XX_1 & \lOnto & \XX_2 & \lOnto & \cdots & \lOnto  & \XX_n &\lOnto &G
\end{diagram}
at each step changing the shade of black and increasing the amount of information provided by the black boxes $\XX_i$.

\begin{figure}[htbp]
\begin{center}
\vspace{.5in}
\includegraphics[scale=.15]{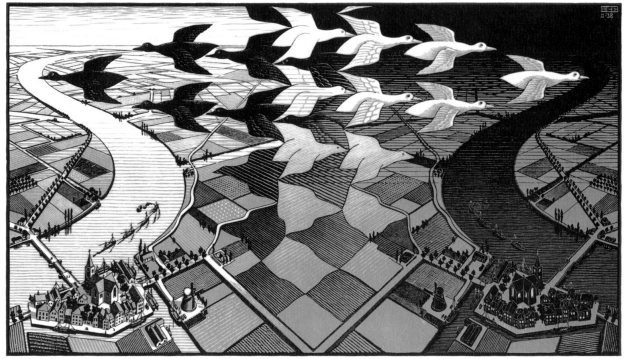}
\caption{ M.C. Escher,  \emph{Day and Night}, 1938}
\label{Escher}
\end{center}
\end{figure}

Step-by-step transformation of black boxes into ``white boxes'' and their complex entanglement is captured well by Escher's famous woodcut, Figure~\ref{Escher}.

\subsection{Randomized algorithms: Monte Carlo and Las Vegas}

This is a brief reminder of two canonical concepts for the benefit of those readers who do not come from a computational group theory background.

A \emph{Monte-Carlo algorithm} is a randomized algorithm which gives a correct output to a decision problem with probability strictly bigger than $1/2$. The probability of having incorrect output can be made arbitrarily small by running the algorithm sufficiently many times. A Monte-Carlo algorithm with outputs ``yes'' and ``no'' is called one-sided if the output ``yes'' is always correct.

A special subclass of Monte-Carlo algorithm is a \emph{Las Vegas algorithm} which either outputs a correct answer or reports failure (the latter with probability less than $1/2$). The probability of having a report of failure is prescribed by the user. A detailed comparison of Monte-Carlo and Las Vegas algorithms, both from practical and theoretical point, can be found in Babai's paper \cite{MR1444127}.

In our setup, Las Vegas makes no sense unless we are given additional information about $X \dashrightarrow G$, for example

\bi
\item generators $x_1,\dots, x_k$ of $\XX$, or
\item the order of $G$, or
\item the isomorphism type of $G$.
\ei

This additional information is frequently available in applications. Even so, it is frequently more practical and therefore preferable to recover the structure of a black box group $\XX$ via a sequence of morphisms
\begin{diagram}
G & \lDotsto^{\pi} & \XX & \lOnto^{\mu_1} & \XX_1 & \lOnto^{\mu_2}  & \cdots & \lOnto^{\mu_n}  & \XX_n &\lOnto^{\mu_{n+1}} &G\\
&& \dOnto_{\rm Id} &&&&&&&& \uOnto_{\rm Id}\\
&& \XX &&&&\relax\rOnto^\lambda &&&& G
\end{diagram}
with
\bi
\item crude but \emph{fast} Monte Carlo morphisms $\mu_i$ for $1 \leqslant i \leqslant n+1$, and
\item a Las Vegas morphism $\lambda$;
\ei
then everything becomes Las Vegas.

Using a ``real world'' simile, we prefer to apply to black box groups the well-known \emph{First Law of Metalworking}: \textbf{use the roughest file first}, see Figure~\ref{fig:files}.

\begin{figure}[h]
\begin{center}
\includegraphics[scale=0.1]{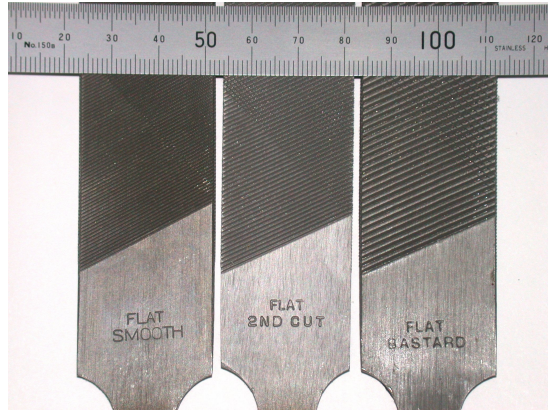}
\end{center}
\caption{Relative tooth sizes for \emph{smooth}, \emph{second cut}, and \emph{bastard} files. There are six different cuts altogether defined as (from roughest to smoothest): \emph{rough, middle, bastard, second cut, smooth}, and \emph{dead smooth}. Source: \emph{Wikipedia}. }
\label{fig:files}
\end{figure}

Even in relatively simple black box problems we may end up dealing with a sophisticated category of black boxes and their  morphisms---this is emphasized in the title of this paper.

\section{Black box fields}

A \emph{black box} (finite) \emph{field} $\KK$ is an oracle or an algorithm operating on $0$-$1$ strings of uniform length (input length), $l(\KK)$, which encrypts a field of known characteristic $p$. The oracle produces random elements from $\KK$ in probabilistic time polynomial in $l(\KK)$, computes $x+y$, $xy$, $x^{-1}$ (for $x \ne 0$) and decides whether $x=y$ for strings $x,y \in \KK$.  We refer the reader to \cite{boneh96.283, maurer07.427} for more details on black box fields and their applications to cryptography.

In this and subsequent paper, we shall be using some results about the isomorphism problem of black box fields \cite{maurer07.427}, that is, the problem of constructing an isomorphism and its inverse between $\KK$ and an explicitly given finite field $\GF_{p^n}$. The explicit data for a finite field of cardinality $p^n$ is defined to be a system of {\emph{structure constants} over the prime field, that is} $n^3$ elements $(c_{ijk})_{i,j,k=1}^n$ of the prime field {$\mathbb{F}_p = \mathbb{Z}/p\mathbb{Z}$ (represented as integers in $[0,p-1]$)} so that $\mathbb{F}_{p^n}$ becomes a field with ordinary addition and multiplication by elements of $\mathbb{F}_p$, and multiplication  determined by
\[
s_i s_j =\sum_{k=1}^n c_{ijk}s_k,
\]
where $s_1, s_2, \dots, s_n$ denotes a basis of $\mathbb{F}_{p^n}$ over $\mathbb{F}_p$. The concept of an explicitly given field of order $p^n$ is robust; indeed, Lenstra Jr.\  has shown in \cite[Theorem 1.2]{lenstra91.329} that for any two fields $A$ and $B$ of order $p^n$ given by two sets of structure constants $(a_{ijk})_{i,j,k=1}^n$ and $(b_{ijk})_{i,j,k=1}^n$ an isomorphism $A \longrightarrow B$ can be constructed in time polynomial in $n\log p$.

Maurer and Raub  \cite{maurer07.427} proved  that the cost of constructing an isomorphism between a black box field $\KK$ and an explicitly given field $\GF_{p^n}$ is reducible in polynomial time to the same problem for the prime subfield in $\KK$ and $\GF_p$.

Using our terminology, their proof can be reformulated to yield the following result.

\begin{theorem} \label{th:bbfiekds}
Let\/ $\KK$ and\/ $\LL$ be black box fields encrypting the same finite field and $\KK_0$, $\LL_0$ their prime subfield. Then a morphism $$\KK_0 \longrightarrow \LL_0$$ can be extended, with the help of polynomial time construction, to a morphism
$$\KK \longrightarrow \LL.$$
\end{theorem}

Obviously, if ${\rm char}\, \KK =p$, we always have a morphism $\GF_p \longrightarrow \KK_0$. The existence of the reverse morphism would follow from solution of the discrete logarithm problem in $\KK_0$. In particular, this means that, for small primes $p$, any two black box fields of the same order $p^n$ are effectively isomorphic.

\section{Automorphisms}

\subsection{Automorphisms as lighter shades of black}\label{50}

The first application of the ``shades of black'' philosophy is the following self-evident theorem which explains how an automorphism of a group can be added to a black box encrypting this group.

\begin{theorem}\label{th:automorphism}
Let $\XX$ be a black box group encrypting a finite group $G$ and assume that each of\/ $k$ tuples of strings
\[
\tilde{x}^{(i)} = (x^{(i)}_1,\dots,x^{(i)}_m),\quad i = 1,\dots,k,
\]
generate $\XX$ in the sense that the projections $\pi\left(x^{(i)}_1\right),\dots,\pi\left(x^{(i)}_m\right)$ generate $G$. Assume that the map
\[
\pi: x^{(i)}_j \mapsto \pi(x^{(i+1 \bmod k)}_j),\quad i = 0,\dots, k-1,\quad j =1,\dots, m,
\]
can be extended to an automorphism $a \in {\rm Aut}\, G$ of order $k$.
The black box group $\YY$ generated in $\XX^k$ by the strings
\[
\bar{x}_j = \left(x^{(0)}_j, x^{(1)}_j, \dots, x^{(k-1)}_j\right), \quad j = 1,\dots, m,
\]
encrypts $G$ via the canonical projection on the first component
\[
(y_0,\dots, y_{k-1}) \mapsto \pi(y_o),
\]
and possess an additional unary operation, cyclic shift
\bea
\alpha: \YY &\longrightarrow& \YY\\
(y_0,y_1,\dots,y_{k-2},y_{k-1}) &\mapsto & (y_1,y_2,\dots,y_{k-1},y_0)
\eea
which encrypts the automorphism $a$ of $G$ in the sense that the following diagram commutes:
\begin{diagram}
\YY &\rOnto^\alpha & \YY\\
\dDotsto &&  \dDotsto\\
G & \rDotsto^a & G
\end{diagram}
\end{theorem}

A somewhat more precise  formulation of Theorem~\ref{th:automorphism} is that we can construct, in time polynomial in $k$ and $m$, $k$ commutative diagrams
\begin{equation}
 \label{diag:Fr}
\begin{diagram}
 \XX & \lOnto^{\{\,\pi_i\,\}_{0\leqslant i\leqslant k-1}} & \XX^k & \lInto^\delta & \YY&\rDotsto^\alpha& \YY\\
 \dDotsto &&\dDotsto && \dDotsto && \dDotsto \\
G & \lDotsto^{\{\,p_i\,\}_{1\leqslant 0\leqslant k-1}} & G^k &\lDotsto^d& G &\rDotsto^a &G
\end{diagram}
\end{equation}
where  $d$ is the twisted diagonal embedding
\bea
d: G & \longrightarrow & G^k \\
 x & \mapsto & (x,x^a,x^{a^2},x^{a^{k-1}}),
\eea
and $p_i$ is the projection
\bea
p_i: G^k & \longrightarrow & G \\
(g_0,\dots,g_i,\dots,g_{k-1}) &\mapsto& g_i.
\eea

Of course, this construction leads to memory requirements increasing by a factor of $k$, but, as our subsequent papers \cite{suko12B,suko12C} show, this is price worth paying. After all, in most practical problems the value of $k$ is not that big, in most interesting cases $k=2$. Also, direct powers of black boxes appear to be very suitable for resorting to parallel computation \cite{Behrends:2010:PCA:1837210.1837239,behrends10.58}.

\subsection{Amalgamation of local automorphisms}

A useful special case of Theorem~\ref{th:automorphism} is the following result about amalgamation of black box automorphisms, stated here in an informal wording rather than expressed by a formal commutative diagram.

\begin{theorem} \label{th:subgroups-automorphisms} Let $\XX$ be a black box group encrypting a group $G$. Assume that $G$ contains subgroups $G_1,\dots, G_l$ invariant under an automorphism $\alpha \in \mathop{{\rm Aut}}G$ and that these subgroups are encrypted in $\XX$ as black boxes $\XX_i$, $i = 1,\dots, l$, supplied with morphisms $$\phi_i: \XX_i \longrightarrow \XX_i$$ which encrypt restrictions $\alpha\!\mid_{G_i}$ of $\alpha$ on $G_i$. Assume also that $\langle G_i, i=1,\dots,l\rangle=G$. Then we can construct, in time polynomial in $l(\XX)$, a morphism $\phi: \XX \longrightarrow \XX$ which encrypts $\alpha$.
\end{theorem}

We shall say in this situation that the automorphism $\phi$ is obtained by amalgamation of local automorphisms $\phi_i$.

We can  generalize Theorem~\ref{th:automorphism} even further.

\begin{theorem} \label{th:subgroups-homomorphisms} Let $\XX$ be a black box group encrypting a group $G$. Assume that $G$ contains subgroups $G_1,\dots, G_l$ mapped by an automorphism $\alpha \in \mathop{{\rm Aut}}G$ to $H_1,\dots, H_l$, respectively, and that these subgroups are encrypted in $\XX$ as black boxes $\YY_i, \ZZ_i$, $i = 1,\dots, l$, supplied with morphisms $$\phi_i: \YY_i \longrightarrow \ZZ_i$$ which encrypt restrictions
\[
\alpha\!\mid_{G_i}: G_i \longrightarrow H_i.
\]
Set $\YY = \langle \YY_i, i=1,\dots,l\rangle$ and $\ZZ = \langle \ZZ_i, i=1,\dots,l\rangle$. Then we can construct, in time polynomial in $l(\XX)$, a morphism $$\phi: \YY \longrightarrow \ZZ$$ which encrypts the restriction of  $\alpha$ to $\langle G_i, i=1,\dots,l\rangle$.
\end{theorem}

Theorem~\ref{th:subgroups-homomorphisms} will be used in the proof of Theorem~\ref{th:inverse-transpose} and \ref{th:sun-in-sln}, and also in the proof of Theorem~\ref{th:G2-in-SL8} in \textbf{\cite{suko12B}}.

\section{Construction of Frobenius maps}
\label{ssec:Frobenius}

We now use Theorems~\ref{th:automorphism} and \ref{th:subgroups-automorphisms}  to construct a Frobenius map on a black box group $\XX$ encrypting $\ppsl_2(q)$.

We give a brief description of a Curtis-Tits system for groups of Lie type of rank at least 3 which is used in the proof of Theorems~\ref{th:frobenius}, \ref{th:inverse-transpose} and \ref{th:sun-in-sln} below. A Curtis-Tits system of a group $G=G(q)$ of Lie type of rank $n$ is a set $\{K_1, \ldots, K_n\}$ of root $\sl_2(q)$-subgroups of $G$ where each $K_i$ corresponds to a node in the Dynkin diagram of $G$. The realtions between $K_i$ and $K_j$, $i\neq j$, are determined by the bonds connecting the nodes, that is, if there is no bond, then $[K_i,K_j]=1$; if there is a single or double bonds, then $\langle K_i,K_j\rangle\cong \sl_3(q)$ or $\ppsp_4(q)$, respectively, see \cite{suko03, suko12A} for more details.

\begin{theorem}[Known characteristic]\label{th:frobenius}
Let $\XX$ be a black box group encrypting a simple Lie type group $G=G(q)$ of untwisted type over a field of order $q=p^k$ for $p$ odd {\rm(}and known{\rm )} and $k >1$. Then we can construct, in time polynomial in $\log |G|$,
\bi
\item a black box $\YY$ encrypting $G$,
\item a morphism $\XX \longleftarrow \YY$, and
\item a morphism $\phi:\YY \longrightarrow \YY$
which encrypts a Frobenius automorphism of $G$ induced by the map $x \mapsto x^p$ on the field $\GF_q$.
\ei
\end{theorem}

\begin{proof}  The proof is based on two applications of Theorem~\ref{th:subgroups-automorphisms}. First we consider the case when $\XX$ encrypts $\psl_2(q)$. Using the standard technique for dealing with involution centralizers, we  can find in $\XX$ a $4$-subgroup $V$; let $E$ be the subgroup in $G=\psl_2(q)$ encrypted by $V$. Since all $4$-subgroups in $\psl_2(q)$ are conjugate to a subgroup in $\psl_2(p)$, we can assume without loss of generality that $E$ belongs to a subfield subgroup $H= \psl_2(p)$ of $G$ and therefore $E$ is centralized by a Frobenius map $F$ on $G$. Now let $e_1$ and $e_2$ be two involutions in $E$, and $C_1$ and $C_2$ maximal cyclic subgroups in their centralizers in $G$; notice that $C_1$ and $C_2$ are conjugate by an element from $H$ and are $F$-invariant.

It follows from basic Galois cohomology considerations that  $F$ acts on $C_1$ and $C_2$ as power maps $\alpha_i : c \mapsto c^{\epsilon p}$ for $p \equiv \epsilon \bmod 4$, $\epsilon=\pm 1$. If now we take images $\XX_i$ of groups $C_i$, we see that the morphisms $\phi_i: x \mapsto  x^{\epsilon p}$ of $\XX_i$ encrypt restrictions of $F$ to $C_i$. Obviously, $\XX_1$ and $\XX_2$ generate a black box $\YY \longrightarrow \XX$, and we can use Theorem~\ref{th:subgroups-automorphisms} to amalgamate $\phi_1$ and $\phi_2$ into a morphism $\phi$ which encrypts $F$.

As usual, for groups $\sl_2(q)$ the same result can be achieved by essentially the same arguments as for $\psl_2(q)$. Moving to other untwisted Chevalley groups, we apply amalgamation to (encryptions of) restrictions of a Frobenius map on $G$ to (encryptions in $\XX$) of a family of root $\ppsl_2$-subgroups $K_i$ in $G$ forming a Curtis-Tits system in $G$ (and therefore generating $G$). Black boxes for Curtis-Tits system in classical groups of odd characteristic are constructed in \cite{suko03}, in exceptional groups in \cite{suko05}. This completes the proof. \end{proof}

\section{Reification of involutions}

Another application of Theorem~\ref{th:automorphism} is a simple, but  powerful procedure which we call ``reification of involutions''.

\subsection{From amalgamation of local automorphisms to reification of involutions}

Let $G$ be a finite group, $a\in\Aut G$ an automorphism of order $2$ and $H \leqslant G$ an $a$-invariant subgroup. We say that the action of $a$ on $H$ is \emph{clean} if $a$ either centralizes $H$ or inverts every elements in $H$.

The following theorem is partly a special case and partly an easy corollary of the amalgamation of local automorphisms, Theorems~\ref{th:subgroups-automorphisms} and \ref{th:subgroups-homomorphisms}.

\begin{theorem}\label{th:reification}
Let $\XX$ be a black box group encrypting a finite group $G$. Assume that $G$ admits an involutive automorphism $a \in \Aut G$ and contains $a$-invariant subgroups $H_1,\dots, H_n$ with a clean action of $a$ on each of them.

Assume also that we are given black boxes $\YY_1,\dots, \YY_n$ encrypting subgroups $H_1,\dots, H_n$. Then we can construct, in polynomial time,
\bi
\item a black box for the structure $\{\,\YY,\alpha\;\}$, where $\YY$ encrypts $H=\langle H_1,\dots, H_n\rangle$ and $\alpha$ encrypts the restriction of $a\mid_H$ of $a$ to $H$;

\item a black box subgroup $\ZZ$ covering $\Omega_1(Z(C_H(a)))$, the subgroup generated by involutions from $Z(C_H(a))$;

 \item if, in addition, the automorphism $a\in G$ and  $H=G$ then $\alpha$ is induced by one of the involutions in $\ZZ$.
 \ei
\end{theorem}

Notice that in many situations (in particular, when $G$ is a Lie type group of odd characteristic), $\ZZ$ is small, and identification of $\alpha$ in $\ZZ$ is easy.

The full strength of reification of involutions will become obvious in \cite{suko14A}.

\subsection{Kantor's and Kassabov's construction of involutions in $\sl_2(2^n)$}

The construction of involutions in $\sl_2(2^n)$ as done by Bill Kantor and Martin Kassabov \cite{kantor-kassabov2013} is the special case of reification of involutions. 

Indeed, let $\XX$ be a black box group encrypting $\sl_2(q)$ with $q=2^n$.

\subsubsection{First construction} Tori of order $q-1$ in $\XX$ are pointwise stabilisers of pairs of points on the projective line over $\GF_q$. Let $g$ and $h$ be two elements of odd orders dividing $q-1$, and assume $[g,h] \ne 1$. If both $g$ and $h$ are contained in the same Borel subgroup of $\XX$ then $[g,h]$ belongs to its unipotent radical and is therefore an involution.

Borel subgroups are stabilisers of points on the projective line. If $g$ and $h$ do not belong to the same Borel subgroup, they belong to some tori $S$ and $T$ whose pairs of fixed points are disjoint, say, $a,b$ and $c,d$. Since $\XX$ acts sharply 3-transitively on the projective line, there is a unique element $w$ in $\XX$ that maps
\[
a \mapsto b, \quad  b \mapsto a,\quad c \mapsto d;
\]
since this permutation contains a $2$-cycle $(a,b)$, it is of even order and is therefore an involution. Now $w$ is an involution that inverts $S$ and  $T$ and can therefore be reified by Theorem~\ref{th:reification}, thus becoming the desired involution. \hfill $\square$

\subsubsection{A shorter construction} However there is a shortcut in the  construction above. Take $f=gh$ and observe that \[
f^w = g^wh^w = g^{-1}h^{-1};
\]
for this calculation, we do not need $w$ as such, its role as a ``virtual'' involution suffices.

We know that
\[
\zeta(f) = f \cdot \sqrt{g\cdot g^w} = gh \cdot \sqrt{ghg^{-1}h^{-1}}
\]
belongs to $C_\XX(w)$; but the latter is a $2$-group, therefore it is an involution or $1$. But
\[
gh \cdot \sqrt{ghg^{-1}h^{-1}} = 1
\]
quickly yields $gh = hg$, which is excluded by our initial  choice of $g$ and $h$.  Hence $\zeta(f)$ is the desired involution. \hfill $\square$

\section{The inverse-transpose map}

In this section, we use Theorems~\ref{th:automorphism} and \ref{th:subgroups-homomorphisms} to construct the inverse-transpose map on $G=\ppsl_n(q)$, $q$ odd.

\subsection{Construction of the inverse-transposed map}

\begin{theorem}\label{th:inverse-transpose}
Let $\XX$ be a black box group encrypting $G=\ppsl_n(q)$, $q$ odd. Then we can construct, in time polynomial in $\log |G|$, a morphism
 \begin{diagram}
 \XX & \rOnto^\varphi & \XX
 \end{diagram}
 that encrypts an inverse-transpose map composed with some inner automorphism of $G$.
\end{theorem}

\begin{proof}
We will prove the result for the black box groups encrypting $\sl_n(q)$, $q$ odd and the result follows from the same arguments for $\psl_n(q)$.

The key observation is that inverse-transpose map is an inner automorphism of $\sl_2(q)$ but not for $\sl_n(q)$ for $n\geq 3$.

Let $G=\sl_2(q)$ and $w=\begin{bmatrix} 0&1\\ -1&0\end{bmatrix}$. Then, conjugation by $w$ is the inverse-transpose map of $G$. Note that $w$ is a Weyl group element corresponding to the diagonal subgroup $T$ (split torus) of $G$, that is, $w$ inverts $T$. In general, the Weyl group elements corresponding to $T$ are of the  form $w_t=\begin{bmatrix} 0&t\\ -t^{-1}&0\end{bmatrix}$ for some $t\in \GF_q^*$. Notice that conjugation by $w_t$ is the composition of the inverse-transpose map with an inner automorphism associated to the diagonal element $\begin{bmatrix} t&0\\ 0&t^{-1}\end{bmatrix}$.

Let $G\cong \sl_n(q)$, $\{K_1, \ldots, K_{n-1}\}$ be a Curtis-Tits system of $G$ and $T$ be the maximal split torus of $G$ normalizing $K_i$ for each $i=1,\ldots, n-1$. Assume also that $w_i\in K_i$ be Weyl group elements which inverts the tori $T_i=T\cap K_i$ for each $i=1,\ldots, n-1$. Then amalgamation on $\{K_1, \ldots, K_{n-1}\}$ the maps
 \begin{diagram}
 K_i & \rOnto^{\varphi_i} & K_i,
 \end{diagram}
where each $\varphi_i$ is a conjugation by $w_i$ on $K_i$, are conjugate to an inverse-transpose map of $K_i$.

A construction of a Curtis-Tits system $\{\KK_1, \ldots, \KK_{n-1}\}$ of a black box group $\XX$ encrypting $\sl_n(q)$ is presented in \cite{suko03} and the construction of the tori $\TT_i$ and $\TT=\langle \TT_1, \ldots, \TT_{n-1}\rangle$  together with Weyl group elements $w_i\in \KK_i$ inverting $\TT_i$ is presented in \cite{suko12A}. Since $\XX$ is generated by $\{\KK_1, \ldots, \KK_{n-1}\}$, applying amalgamation to the conjugation by $w_i \in \KK_i$, we have an automorphism of $\XX$ which encrypts the inverse transpose map composed with some inner automorphism of $G$.
\end{proof}

\subsection{An application: $\su_n(q) \hookrightarrow \sl_n(q^2)$}

We apply our arguments for the construction of $\su_n(q)$ inside $\sl_n(q^2)$. We note that the centralizer of inverse transpose map composed with Frobenius map in $\sl_n(q^2)$ is the subgroup $\su_n(q)$.

\begin{theorem}[Known characteristic] \label{th:sun-in-sln} \
Let $\XX$ be a black box group encrypting the group $\sl_n(q^2)$ for $q$ odd, $q=p^k$ for some $k$ (perhaps unknown) and a known prime number $p$. Then we can construct, in time polynomial in $\log q$ and $n$, a black box group $\YY$ encrypting the group $\su_n(q)$ and a morphism $\YY \hookrightarrow \XX$.
\end{theorem}

An important feature of the proof of this theorem (and other similar results in \cite{suko12B}) is that they never refer to the ground fields of the groups and do not involve any computations with unipotent elements. In fact, we interpret morphisms between functors
\[
\su_n(\cdot) \hookrightarrow \sl_n(\cdot^2).
\]
within  our black boxes. At a practical level, it means that given a black box group encrypting $\sl_3(p^2)$ for a 60 decimal digits long prime number $p$, say
\[
p=622288097498926496141095869268883999563096063592498055290461
\]
(one of the examples run in GAP on a pretty old and underpowered laptop computer), we can construct a black box subgroup
\[
\su_3(p) \hookrightarrow  \sl_3(p^2).
\]

This example shows that a modicum of categorical language is useful for the theory as well as for its implementation in the code since it suggests a natural structural approach to development of the computer code.

\begin{proof}[Proof of Theorem \ref{th:sun-in-sln}]
Let $\{\KK_1, \ldots, \KK_{n-1}\}$ be a Curtis-Tits system for $\XX$. Assume also that, for each $i=1,2,\ldots, n-1$, $\TT_i$ is a maximal split torus of $\KK_i$ where $\langle \TT_1, \ldots, \TT_{n-1} \rangle$ encrypts a maximal split torus of $G$ normalizing each $\KK_i$, and $w_i\in \KK_i$ are the Weyl group elements, that is, $w_i$ inverts $\TT_i$. Furthermore, set $W_i=C_{\KK_i}(w_i)$ and $\TT_{w_i}$ is the maximal torus in $W_i$ for each $i=1,2,\dots,n-1$.

Let $\phi_i$ denote the restriction of the Frobenius map on $\KK_i$.

Since the inverse transpose map is amalgamated over the conjugation by Weyl group elements $w_i$, the inverse transpose map inverts $\TT_i$ and fixes $\TT_{w_i}$. Moreover the restriction of the Frobenius map $\phi_i$ acts on $\TT_i$ and $\TT_{w_i}$ as $x\mapsto x^{\epsilon q}$ for $q \equiv \epsilon \mbox{ mod } 4$, $\epsilon=\pm 1$. Therefore, the centralizer of $w_i\circ \phi_i$ consists of the elements from $\KK_i$ satisfying
\[
t=t^{-\epsilon q} \mbox{ or, equivalently } t^{\epsilon q+1}=1 \mbox{ for } t\in \TT_i,
\]
and
\[
t=t^{\epsilon q} \mbox{ or, equivalently } t^{\epsilon q-1}=1 \mbox{ for } t \in \TT_{w_i}.
\]
Since $\KK_i=\langle \TT_i,\TT_{w_i} \rangle$, we can compute $\CC_i:=C_{\KK_i}(w_i\circ \phi_i)$. Moreover, since $\XX =\langle \KK_1, \ldots, \KK_{n-1} \rangle$,  amalgamation over $\CC_i$ gives the subgroup encrypting the fixed points of the inverse transpose map composed with Frobenius map which is isomorphic to $\su_n(q)$.
\end{proof}

\section{Structure recovery}
\label{sec:oracle}

In this section, we revise the classification of black box group problems and outline our vision of the hierarchy of typical black box group problems.

\label{ssec:matrix-project}

\subsection{The hierarchy of black box problems}

\begin{description}
\item[\textbf{Verification Problem}] Is the unknown group encrypted by a black box group $\XX$ isomorphic to the given group $G$ (``target group'')?

\item[\textbf{Recognition Problem}] Determine the isomorphism class of the group encrypted by $\XX$.
\end{description}

The Verification  Problem arises as a sub-problem within more complicated Recognition Problems. The two problems have dramatically different complexity. For example, the celebrated Miller-Rabin algorithm \cite{rabin80.128} for testing primality of the given odd number $n$ is simply a  black box algorithm for solving the verification problem for the multiplicative group   $\mathbb{Z}/n\mathbb{Z}^*$ of residues modulo $n$ (given by a simple black box: take your favorite random numbers generator and generate random integers between $1$ and $n$) and the cyclic group $\mathbb{Z}/(n-1)\mathbb{Z}$ of order $n-1$ as the target group. On the other hand, if $n=pq$ is the product of primes $p$ and $q$, the recognition problem for the same black box group means finding the direct product decomposition
\[
\mathbb{Z}/n\mathbb{Z}^* \cong (\mathbb{Z}/(p-1)\mathbb{Z}) \oplus (\mathbb{Z}/(q-1)\mathbb{Z})
\]
which is equivalent to factorization of $n$ into product of primes.

The next step after finding the isomorphism type of the black box group $\XX$ is

\begin{description}
\item[Constructive Recognition] Suppose that a black box group $\XX$ encrypts a concrete and explicitly given group $G$. Rewording a definition given in \cite{brooksbank08.885},
    \begin{quote}
        \emph{ The goal of a constructive recognition algorithm is
to construct an effective isomorphism $\Psi: G \longrightarrow X$. That is, given $g\in G$, there is an efficient procedure to construct a string $\Psi(g)$ encrypting $g$  in $\XX$ and given a string $x$  produced by $\XX$, there is an efficient procedure to construct the element $\Psi^{-1}(x) \in G$ encrypted by $\XX$.}
    \end{quote}
\end{description}

However, there are still no really efficient constructive recognition algorithms for black box groups $\XX$ of (known) Lie type over a finite field of large order $q=p^k$. The first computational obstacles for known algorithms
\cite{brooksbank01.79,brooksbank03.162,brooksbank08.885,brooksbank01.95,brooksbank06.256,celler98.11,conder06.1203,leedham-green09.833} are the need to construct unipotent elements in black box groups, \cite{brooksbank01.79,brooksbank03.162,brooksbank08.885,brooksbank06.256,brooksbank01.95,celler98.11} or to solve discrete logarithm problem for matrix groups \cite{conder01.113,conder06.1203,leedham-green09.833}.

Unfortunately, the probability that the order of a random element is divisible by $p$ is $O(1/q)$ \cite{guralnick01.169}, so one has to make  $O(q)$ (that is, \emph{exponentially many}, in terms of the input length $O(\log q)$ of the black boxes and the algorithms) random selections of elements in a given group to construct a unipotent element. However, this brute force approach is still working for small values of $q$, and  Kantor and Seress \cite{kantor01.168} used it to develop an algorithm for recognition of black box classical groups. Later the algorithms of \cite{kantor01.168} were upgraded to polynomial time constructive recognition algorithms \cite{brooksbank03.162, brooksbank08.885, brooksbank01.95, brooksbank06.256} by assuming the availability of additional \emph{oracles}:
 \begin{itemize}
 \item the \emph{discrete logarithm oracle} in $\mathbb{F}_q^*$, and
  \item the  \emph{$\sl_2(q)$-oracle}.
   \end{itemize}
Here, the \emph{$\sl_2(q)$-oracle} is a procedure for constructive recognition of $\sl_2(q)$; see  discussion in \cite[Section~3]{brooksbank08.885}.

\begin{quote}
\textbf{\emph{We emphasize that in this and subsequent papers we are using neither the discrete logarithm oracle in $\mathbb{F}_q^*$ nor the  $\sl_2(q)$-oracle.}}
\end{quote}

\subsection{Structure recovery}
\label{ssec:structure-recovery}

Suppose that a black box group $\XX$ encrypts a concrete and explicitly given group $G=G(\GF_q)$ of Chevalley type $G$ over a explicitly given finite field $\GF_q$. To achieve \emph{structure recovery} in $\XX$ means to construct, in probabilistic polynomial time in $\log |G|$,
    \bi
     \item a black box field $\KK$ encrypting $\GF_q$, and
     \item a probabilistic polynomial time morphism $$\Psi: G(\KK) \longrightarrow \XX.$$
    \ei

 We extend our definition from \cite{suko12A} where it refers to a special case of the present one.
 
 An example of a structure recovery is presented in our next paper \cite{suko14A} for black box groups $\ppsl_2(q)$.

\subsection{Separation of flesh from bones}
Recall that simple algebraic groups (in particular, Chevalley groups over finite fields) are understood in the theory of algebraic groups as functors from the category of unital commutative rings into the category of groups; most structural properties of a Chevalley group are encoded in the functor; the field mostly provides the flesh on the bones.

A ``category-theoretical'' approach  allows us to carry out constructions like the following one.

\begin{theorem}[Known characteristic \textbf{\cite{suko12B}}]
\label{th:G2-in-SL8}
Let $\XX$ be a black box group encrypting the group $\sl_8(F)$ for a field $F$ of (unknown) odd order $q = p^k$ but known $p= {\rm char}\,F$.  Then we can construct, in time polynomial in $\log |F|$, a chain of black box groups and morphisms
\[
\UU \hookrightarrow \VV \hookrightarrow \WW \hookrightarrow \XX
\]
that encrypts the chain of canonical embeddings
\[
\mathop{G_2}(F) \hookrightarrow \so_7(F) \hookrightarrow\so_8^+(F) \hookrightarrow \sl_8(F).
\]
\end{theorem}

Again, these constructions  (and even the embedding \[^3{\rm D}_4(q) \hookrightarrow  \so_8^+(q^3),\]  also done in \cite{suko12B}) are ``field-free'' and, moreover, ``characteristic-free''.

Another aspect of the concept of ``structure recovery'' is that it follows an important technique from model-theoretic algebra: interpretability of one algebraic structure in another, see, for example, \cite{borovik94-BN}. Construction of a black box field in a black box group in \cite{suko14A} closely follows this model-theoretic paradigm.

\section*{Acknowledgements}

This paper would have never been written if the authors did not enjoy the warm hospitality offered to them at the Nesin Mathematics Village (in \c{S}irince, Izmir Province, Turkey) in August 2011, August 2012, and July 2013; our thanks go to Ali Nesin and to all volunteers and staff who have made the Village a mathematical paradise.

Final touches to the paper were put during the Workshop on Model Theory and Groups at \.{I}stanbul Center for Mathematical Sciences in March 2014.

We thank Adrien Deloro for many fruitful discussions, in \c{S}irince and elsewhere, and
Bill Kantor, Charles Leedham-Green, and Rob Wilson for their helpful comments.

Special thanks go to our logician colleagues: Paola D'Aquino, Gregory Cherlin, Emil  Je{\v{r}}{\'a}bek, Jan Kraj{\'{\i}}{\v{c}}ek,  Angus Macintyre, Jeff Paris, Jonathan Pila,  Alex Wilkie, and Boris Zilber  for pointing to fascinating connections with logic and complexity theory.

Our work was partially supported by the Marie Curie FP7 Initial Training Network MALOA (PITN-GA-2008-MALOA no. 238381).

We gratefully acknowledge the use of Paul Taylor's \emph{Commutative Diagrams} package, \url{http://www.paultaylor.eu/diagrams/}.

\end{document}